\documentclass[a4paper,10pt]{amsart}
\usepackage[utf8x]{inputenc}

\newtheorem{thm}{Theorem}
\newtheorem*{ack}{Acknowledgements}

\newtheorem{lm}[thm]{Lemma}
\newtheorem{plm}{Problem}
\newtheorem{coro}[thm]{Corollary}
\newtheorem{rmq}[thm]{Remark}
\numberwithin{thm}{section}
\numberwithin{equation}{section}
\newcommand{\s}[1]{\sigma_{ext}(#1)}
\DeclareMathOperator{\Span}{span}
\title[Extended eigenvectors of $S_B$]{On extended eigenvalues and extended eigenvectors of truncated shift}
\author{HASAN ALKANJO}
\address{Universit\'e de Lyon; Universit\'e Lyon 1; Institut Camille Jordan CNRS UMR 5208; 43, boulevard du 11 Novembre 1918, F-69622 Villeurbanne}
\email{alkanjo@math.univ-lyon1.fr}
\begin{document}

\begin{abstract}
We give a complete description of the set of extended eigenvectors of truncated shifts defined on the model spaces
$K_u^2:=H^2\ominus uH^2$, in the case of $u$ is a Blaschke product.
\end{abstract}
\maketitle
\section{Introduction and preliminaries}
Let $H$ be a complex Hilbert space, and denote by $\mathcal{L}(H)$ the algebra of all bounded linear operators on $H$.
If $T$ is an operator in $\mathcal{L}(H)$, then a complex number $\lambda$ is an extended eigenvalue of $T$ if there is a nonzero
operator $X$ such that $TX=\lambda XT$. We denote by the symbol $\sigma_{ext}(T)$ the set of extended eigenvalues of $T$.
The set of all extended eigenvectors corresponding to $\lambda$ will be denoted as $E_{ext}(\lambda)$.
Obviously $1\in\sigma_{ext}(T)$ for any operator $T$. Indeed, one can take $X$ being the identity operator.

Let $T$ in $\mathcal{L}(H)$, and let $\sigma(T)$ and $\sigma_p(T)$ denote the spectrum and the point spectrum of $T$ respectively.
By a theorem of Rosenblum \cite{ros}, it was established in \cite{pfini} that
\begin{equation}\label{inc}
 \sigma_{ext}(T)\subset\{\lambda\in\mathbb{C} : \sigma(T)\cap\sigma(\lambda T)\neq\emptyset\}.
\end{equation}

Moreover, when $H$ is finite dimensional, in \cite{pfini} the set of extended eigenvalues has been characterized by the
following theorem
\begin{thm}\label{fini}
 Let $T$ be an operator on a finite dimensional Hilbert space $H$.
Then $\sigma_{ext}(T)=\{\lambda\in\mathbb{C} : \sigma(T)\cap\sigma(\lambda T)\neq\emptyset\}.$
\end{thm}
 \begin{proof}
   First we consider the case when $T$ is not invertible. In this situation
both $T$ and $T^*$ have nontrivial kernels. Let $X^\prime$ be a nonzero operator from kernel of $T^*$
to kernel of $T$. Define $X=X^\prime P$ where $P$ denotes the orthogonal projection on kernel of $T^*$.
Clearly, $X\neq0$, and $TX=0=\lambda XT$ for any $\lambda\in\mathbb{C}$. Consequently, $\sigma_{ext}(T)=\mathbb{C}$.
On the other hand, since $T$ is not invertible, for any complex number $\lambda$, $0\in\sigma(T)\cap\sigma(\lambda T)$.
Thus $$\s{T}=\mathbb{C}=\{\lambda\in\mathbb{C} : \sigma(T)\cap\sigma(\lambda T)\neq\emptyset\}.$$

Now assume that $T$ is invertible so that $0\notin\sigma(T)$. In view of (\ref{inc}) it suffices to show that
$\{\lambda\in\mathbb{C} : \sigma(T)\cap\sigma(\lambda T)\}\subset\s{T}.$
So suppose that $\alpha$ is a (necessarily nonzero) complex number such that $\alpha\in\sigma(T)$ and
$\alpha\in\sigma(\lambda T)$. Since $\alpha\in\sigma(T)$ there exists a vector $a$ such that $Ta=\alpha a$.
On the other hand, $\alpha\in\sigma(\lambda T)$ implies that $\lambda\neq 0$ so $\alpha/\lambda\in\sigma(T)$.
Therefore, $(\overline{\alpha/\lambda})\in\sigma(T^*)$ and there is a vector $b$ such that $T^*b=(\overline{\alpha/\lambda})b$.
Let $X=a\otimes b$. Then $TX=\lambda XT$ and consequently $\lambda\in\sigma(T)$.
 \end{proof}
From this theorem it derives the following consequences
\begin{coro}\label{fini1}
 Let $T$ be an operator on a finite dimensional Hilbert space $H$. Then
\begin{enumerate}
 \item If T is invertible then $\sigma_{ext}(T)=\{\alpha/\beta : \alpha,\beta\in\sigma(T)\}$, and
if $Ta=\alpha a$, $T^*b=\overline{\beta}b$ then $a\otimes b\in E_{ext}(\alpha/\beta)$.
 \item $\sigma_{ext}(T)=\{1\}$ if and only if $\sigma(T)=\{\alpha\}, \alpha\neq0$.
 \item $\sigma_{ext}(T)=\mathbb{C}$ if and only if $0\in\sigma(T)$. Moreover, this assertion remains available in
infinite dimensional Hilbert spaces if $0\in\sigma_p(T)\cap\sigma_p(T^*)$.
\end{enumerate}

\end{coro}

The next section contains the needed background on the spaces $K_u^2$.
\begin{section}{background on $K_u^2$}
 Nothing in the section is new, and the bulk of it can be found in standard sources, for example \cite{nik},
\cite{berc}, \cite{foi} and \cite{don}.
\begin{subsection}{Basic notation, model spaces and kernel functions}
 Let $H^2$ be the standard Hardy space, the Hilbert space of holomorphic functions in the open unit disk
$\mathbb{D}\subset\mathbb{C}$ having square-summable Taylor coefficients at the origin. We let $S$ denote the
unilateral shift operator on $H^2$. Its adjoint, the backward shift, is given by
\begin{equation}
 S^*f(z)=\frac{f(z)-f(0)}{z}.
\end{equation}
For the remainder of the paper, $u$ will denote a non-constant inner function. the subspace $K_u^2=H^2\ominus uH^2$
is a proper nontrivial invariant subspace of $S^*$, the most general one by the well-known theorem of A. Beurling.
The compression of $S$ to $K_u^2$ will be denoted by $S_u$. Its adjoint, $S_u^*$, is the restriction of $S^*$ to $K_u^2$.
For $\lambda$ in $\mathbb{D}$, the kernel function in $H^2$ for the functional of evaluation at $\lambda$
will be denoted by $k_\lambda$; it is given explicitly by
\begin{equation}
 k_\lambda(z)=\frac{1}{1-\overline{\lambda}z}.
\end{equation}
 Letting $P_u$ denote the orthogonal projection from $L^2$ onto $K_u^2$. The kernel function in $K_u^2$
for the functional of evaluation at $\lambda$ will be denoted by $k_\lambda^u$. It is natural that $k_\lambda^u$
equals $P_u k_\lambda$, i.e.,
\begin{equation}
 k_\lambda^u(z)=\frac{1-\overline{u(\lambda)}u(z)}{1-\overline{\lambda}z}.
\end{equation}
\end{subsection}

\begin{subsection}{Riesz bases of $K_u^2$}
 It is known that the model space $K_u^2$ is finite dimensional if and only if $u$ is finite Blaschke product
\begin{equation}\label{blf}
 B(z)=\prod_{i=1}^{n}b_{\alpha_i}^{p_i},\ \ with\ \ 
b_{\lambda}=\frac{\lambda-z}{1-\overline{\lambda}z}\ \ for\ \lambda\in\mathbb{D},\ \ p_i,n\in\mathbb{N^*},\
and\ \alpha_i\neq\alpha_j\ for\ i\neq j.
\end{equation}

In the general case, if $B$ is an infinite Blaschke product defined by
\begin{equation}\label{blinf}
B(z)=\prod_{i=1}^{\infty}\frac{|\alpha_i|}{\alpha_i } b_{\alpha_i}^{p_i},\  p_i\in\mathbb{N^*},
\end{equation}
then the following Cauchy kernels
\begin{equation}
 e_{i,l}(z)=\frac{l!z^l}{({1-\overline{\alpha_i}z})^{l+1}},\ \forall i\geq1,\ \ l=0,...,p_i-1,
\end{equation}
span the space $K_B^2$.
In particular, if $p_i=1$ for $i$ in $\mathbb{N}^*$, then $e_{i,0}$ will be denoted by $e_i$, i.e.,
\begin{equation}
 e_i(z)=k_{\alpha_i}^B(z).
\end{equation}
If we denote by $\{e_{i,l}^* : i\geq1, \ l=0,...,p_i-1\}$ (see \cite{don}) the dual set of 
$\{e_{i,l} : i\geq1, \ l=0,...,p_i-1\}$,
(i.e., the set of kernels verifying
\begin{equation}\label{cronic}
 \langle e_{i,k}^*,e_{j,l}\rangle=\delta_{ij}\delta_{kl},\ \forall\ i,j\geq1,\ k=0,...,p_i-1,\ l=0,...,p_j-1,
\end{equation}
where $\langle.,.\rangle$ denotes the inner product in $L^2$, and $\delta_{ij}$ denotes the well-known Kronecker 
$\delta-$symbol), then we have the following lemma
\begin{lm}\label{eigenvect}
 If $B$ is a Blaschke product defined by (\ref{blinf}), then
\begin{equation*}
 S_B^*e_{i,l} = \left\{
  \begin{array}{l l}
    \overline{\alpha_i}e_{i,0} & \quad \text{if $l=0$}\\
    le_{i,l-1}+\overline{\alpha_i}e_{i,l} & \quad \text{otherwise,}\\
  \end{array} \right.
\end{equation*}
and
\begin{equation*}
 S_Be^*_{i,l} = \left\{
  \begin{array}{l l}
    \alpha_i e^*_{i,p_i-1} & \quad \text{if $l=p_i-1$}\\
    \alpha_i e^*_{i,l}+(l+1)e^*_{i,l+1} & \quad \text{otherwise.}\\
  \end{array} \right.
\end{equation*}
\end{lm}
\begin{proof}
 For the first equality, if $l=0$, then
$$S_B^*e_{i,0}(z)=\frac{k_{\alpha_i}^B(z)- k_{\alpha_i}^B(0)}{z}=\frac{\overline{\alpha_i}}{1-\overline{\alpha_i}z }=
\overline{\alpha_i} e_{i,0}(z).$$
Otherwise,
$$S_B^*e_{i,l}(z)=\frac{l!z^{l-1}}{({1-\overline{\alpha_i}z})^{l+1}}=
l!(\frac{z^{l-1}}{({1-\overline{\alpha_i}z})^l}+\overline{\alpha_i}\frac{z^l}{({1-\overline{\alpha_i}z})^{l+1}})$$
$$=le_{i,l-1}(z)+\overline{\alpha_i}e_{i,l}(z).$$
For the second equality, it is sufficient to use the first one together with the fact that
$$\langle S_Be_{i,k}^*,e_{j,l}\rangle=\langle e_{i,k}^*,S_B^*e_{j,l}\rangle,\
\forall\ i,j\geq1,\ k=0,...,p_i-1,\ l=0,...,p_j-1.$$
\end{proof}
If we denote by $E_i=\Span\{e_{i,0},...,e_{i,p_i-1}\}$ and by $E_i^*=\Span\{e^*_{i,0},...,e^*_{i,p_i-1}\}$,
for $i$ in $\mathbb{N}^*$. Then Lemma \ref{eigenvect} derives the following consequences
\begin{coro}\label{vect}
 For each $i$ in $\mathbb{N}^*$, we have 
\begin{enumerate}
 \item The subspaces $E_i$ and $E_i^*$ are invariant of $S_B^*$ and $S_B$ respectively.
 \item Let $l\in\{0,1,...,p_i-1\} $. For each $k=0,1,...,l$, we have 
$$(S_B-\alpha_iI)^{k}e^*_{i,p_i-l-1}\neq0,\  and\  (S_B-\alpha_iI)^{l+1}e^*_{i,p_i-l-1}=0.$$
In particular, $\ker(S_B-\alpha_iI)^{l+1}=\Span\{e^*_{i,p_i-l-1},...,e^*_{i,p_i-1}\}$, and for all $k\geq p_i$, 
we have $\ker(S_B-\alpha_iI)^{k}=\ker(S_B-\alpha_iI)^{p_i}=E_i^*.$
\end{enumerate}
 
\end{coro}
\begin{proof}
 The first point is trivial. For the second one, we will argue by induction. This result is trivial for $l=0$.
We assume that it is true for all $k=0,1,...,l-1$, i.e.,
$$x:=(S_B-\alpha_iI)^{l-1}e^*_{i,p_i-l}\neq0,\  and\  (S_B-\alpha_iI)x=0.$$
It is enough to show that
$$(S_B-\alpha_iI)^{l}e^*_{i,p_i-l-1}\neq0,\  and\  (S_B-\alpha_iI)^{l+1}e^*_{i,p_i-l-1}=0.$$
By using Lemma \ref{eigenvect} and the induction hypothesis, we have that
$$(S_B-\alpha_iI)^{l}e^*_{i,p_i-l-1}=(p_i-l)x\neq0,$$
and
$$(S_B-\alpha_iI)^{l+1}e^*_{i,p_i-l-1}=(p_i-1)(S_B-\alpha_iI)x=0.$$
Consequently, $\Span\{e^*_{i,p_i-l-1},...,e^*_{i,p_i-1}\}\subset \ker(S_B-\alpha_iI)^{l+1}$ and $(S_B-\alpha_iI)^{l+1}$
is injective on $\Span\{e^*_{i,p_i-l-1},...,e^*_{i,p_i-1}\}$
To complete the proof, we shall show that $(S_B-\alpha_iI)^{l+1}$ is injective on 
$$\Span\{E_j^*:j\geq1\ and\ j\neq i\}.$$
But the subspaces $E_j^*$ are invariant of $(S_B-\alpha_iI)^{l+1}$. Thus, it is sufficient to show that 
$(S_B-\alpha_iI)^{l+1}$ is injective on $E_j^*$ for any $j\neq i$.
To do so, suppose to the contrary that $(S_B-\alpha_iI)^{l+1}x=0$ for $x\in E_j^*$ and $j\neq i$, then 
$(S_B-\alpha_iI)^{l}x\in \Span\{e^*_{i,p_i-1}\}$, which contradicts the fact that $E_j^*$ is 
invariant of $(S_B-\alpha_iI)^{l}$.

\end{proof}

Biswas and Petrovic determine in \cite{pfini} the extended spectrum of truncated shift.
Our main result, that is Theorem \ref{moi2}, gives a complete description of the set of extended eigenvectors of 
truncated shift $S_B$. Moreover, it affirms the result of Biswas and Petrovic for the set $\s{S_B}$ without
using the Sz.-Nagy-Foias commutant lifting theorem. Consequently, it strengthens \cite[Theorem 3.10]{pfini}.
\end{subsection}
\end{section}
\begin{section}{Extended eigenvalues and extended eigenvectors of $S_B$}
If $B$ is a Blaschke product defined by (\ref{blinf}), it was shown in \cite{nik} that 
$\sigma(S_B)=\overline{\{\alpha_i\}}_{i\geq1}$, and $\sigma_p(S_B)= \{\alpha_i\}_{i\geq1}$.
For the remainder of this paper, the zeros$\{\alpha_i\}_{i\geq1}$ are all nonzero. Before showing our main result,
we give theorem \ref{moi1} as a direct application of Theorem \ref{inc} and Lemma \ref{eigenvect}.
If $B$ is a finite Blaschke product defined by (\ref{blf})
with $p_i=1$ for all $i$,
then by Corollary \ref{fini1}, $\s{S_B}=\{\alpha_i/\alpha_j : i,j=1...n\}$ and
$e^*_i\otimes e_j\in E_{ext}(\alpha_i/\alpha_j)$.
It is natural to ask weather this eigenvector is unique or not.
The following theorem answers this question affirmatively.

\begin{thm}\label{moi1}
 If $B$ is a finite Blaschke product defined in (\ref{blf}) with $p_i=1$ for all $i$, then
$\s{S_B}=\{\alpha_i/\alpha_j : i,j=1,...,n\}$ and $E_{ext}(\alpha_i/\alpha_j)=
\Span\{e^*_k\otimes e_l : \alpha_k/\alpha_l=\alpha_i/\alpha_j\}$.
\end{thm}
\begin{proof} Since $\{e_i\}_{i=1}^n$ and $\{e^*_i\}_{i=1}^n$ are bases Riesz for $K_B^2$,
the set $\{E_{ij}:=e^*_i\otimes e_j\}_{i,j=1}^n$ is a basis Riesz for $\mathcal{L}(K_B^2)$. Now assume that $X\in \mathcal{L}(K_B^2)$
is a solution to the equation
\begin{equation*}
 S_BX=\frac{\alpha_i}{\alpha_j}XS_B,
\end{equation*}
then there are a family of complex numbers $\{a_{ij}\}_{i,j=1}^n$ such that
\begin{equation*}
 S_B(\sum_{k,l=1}^na_{kl}E_{kl})=\frac{\alpha_i}{\alpha_j}(\sum_{k,l=1}^na_{kl}E_{kl})S_B,
\end{equation*}
hence
\begin{equation*}
 (\sum_{k,l=1}^n\frac{\alpha_k}{\alpha_l}a_{kl}E_{kl})S_B=(\sum_{k,l=1}^n\frac{\alpha_i}{\alpha_j}a_{kl}E_{kl})S_B,
\end{equation*}
Since $S_B$ is invertible and $\{E_{ij}:=e^*_i\otimes e_j\}_{i,j=1}^n$ is a Riesz basis for $\mathcal{L}(K_B^2)$,
\begin{equation*}
 \frac{\alpha_k}{\alpha_l}a_{kl}=\frac{\alpha_i}{\alpha_j}a_{kl},\ \forall k,l=1,...,n,
\end{equation*}
thus $$E_{ext}(\frac{\alpha_i}{\alpha_j})=
\Span\{e^*_k\otimes e_l : \frac{\alpha_k}{\alpha_l}=\frac{\alpha_i}{\alpha_j}\}.$$
\end{proof}

\begin{rmq}
if  $\alpha_k/\alpha_l\neq\alpha_i/\alpha_j$ for all $(k,l)\neq(i,j) $, then
$$E_{ext}(\frac{\alpha_i}{\alpha_j})=\{e^*_i\otimes e_j\},$$ that is why we have said that this solution is unique.

\end{rmq}
Now, let $B$ be an infinite Blaschke product as in (\ref{blinf}), and let $\{\gamma_i\}_{i\in I}$ be the set of
limit points of $\{\alpha_i\}_{i\geq1}$ on the circle $\mathbb{T}=\{z\in\mathbb{C} : |z|=1\}$. By (\ref{inc}), we have
$$\s{S_B}\subset\{\frac{\alpha_i}{\alpha_j} : i,j\geq1\}\cup\{\frac{\alpha_i}{\gamma_j} : i\geq1,\ j\in I\}
\cup\{\frac{\gamma_i}{\alpha_j} : i\in I,\ j\geq1\}.$$
The following theorem shows that this inclusion is proper, more precisely
\begin{thm}\label{moi2}
 If $B$ is an infinite Blaschke product defined by (\ref{blinf}), then
$$\s{S_B}=\{\frac{\alpha_i}{\alpha_j} : i,j\geq1\},$$
and for any $i,j\geq1$, we have
$$E_{ext}(\frac{\alpha_i}{\alpha_j})=
\Span\{\sum_{k=0}^{l}(\sum_{r=0}^{k}c_{k-r}(\frac{\alpha_m}{\alpha_n})^r
\frac{(l+r-k)!(p_m-r-1)!}{(l-k)!(p_m-1)!} e^*_{m,p_m-r-1})\otimes e_{n,l-k}$$
$$\forall m,n\geq1\ where\ \frac{\alpha_m}{\alpha_n}=\frac{\alpha_i}{\alpha_j},\ l=0,...,\min(p_m-1,p_n-1),\ 
c_{k-r}\in\mathbb{C}\ and\ c_0\neq0\}.$$
\end{thm}
\begin{proof}
 Let $\lambda\in\mathbb{C}$ and $X\in\mathcal{L}(K_B^2)$ be such that
$$S_BX=\lambda XS_B,$$ then by Lemma \ref{eigenvect}, for all $j\geq1$ we have
\begin{equation*}
 S_BXe^*_{j,l} = \left\{
  \begin{array}{l l}
    \lambda\alpha_j Xe^*_{j,p_j-1} & \quad \text{if $l=p_j-1$}\\
    \lambda\alpha_j Xe^*_{j,l}+\lambda(l+1)Xe^*_{j,l+1} & \quad \text{if $l=0,...,p_j-2$.}\\
  \end{array} \right.
\end{equation*}
If $X\neq0$, then necessarily there are $i,j\geq1$, $l$ in $\{0,1,...,p_j-1\}$ and $(c_{0}\neq0)$ in $\mathbb{C}$ 
such that
$$\lambda=\frac{\alpha_i}{\alpha_j}\   and\   Xe^*_{j,l}=c_{0}e^*_{i,p_i-1}.$$
Then
\begin{equation}\label{q1}
 S_BXe^*_{j,l-1}=\frac{\alpha_i}{\alpha_j}X(\alpha_je^*_{j,l-1}+le^*_{j,l}),
\end{equation}
$$(S_B-\alpha_iI)Xe^*_{j,l-1}=\frac{\alpha_i}{\alpha_j}lc_{0}e^*_{i,p_i-1},$$
consequently there exist complex numbers $(c_{0}^{(1)}\neq0)$ and $c_1$ such that
$$Xe^*_{j,l-1}=c_{0}^{(1)}e^*_{i,p_i-2}+c_1e^*_{i,p_i-1},$$
moreover, by (\ref{q1})
$$c_{0}^{(1)}(\alpha_ie^*_{i,p_i-2}+(p_i-1)e^*_{i,p_i-1})+c_1\alpha_ie^*_{i,p_i-1}=
\alpha_i(c_{0}^{(1)}e^*_{i,p_i-2}+c_1e^*_{i,p_i-1})+\frac{\alpha_i}{\alpha_j}lc_{0}e^*_{i,p_i-1},$$
hence
$$c_{0}^{(1)}=\frac{\alpha_i}{\alpha_j}\frac{l}{p_i-1}c_{0}.$$
By repeating the same calculation a number of times equal to $min(p_i-2,l-1)$, we obtain that
$$Xe^*_{j,l-k}=\sum_{r=0}^{k}c_{k-r}^{(r)}e^*_{i,p_i-r-1},\  where$$
$$c_{k-r}^{(r)}=(\frac{\alpha_i}{\alpha_j})^r\frac{(l+r-k)!(p_i-r-1)!}{(l-k)!(p_i-1)!}c_{k-r},\ k=2,...,min(p_i-1,l),$$
thus, if $l\geq p_i$, we have
$$(S_B-\alpha_iI)Xe^*_{j,l-p_i}=\frac{\alpha_i}{\alpha_j}(l-p_i+1)
\sum_{r=0}^{p_i-1}c_{p_i-1-r}^{(r)}e^*_{i,p_i-1-r},\ where\ c_{0}^{(p_i-1)}\neq0,$$
therefore
$$(S_B-\alpha_iI)^{p_i}Xe^*_{j,l-p_i}\neq0\ and\
(S_B-\alpha_iI)^{p_i+1}Xe^*_{j,l-p_i}=0,$$
and that contradicts Corollary \ref{vect}. Thus, if $\lambda=\frac{\alpha_i}{\alpha_j}$ and $X\neq0$, then
$l$ must be in the range
$\{0,1,...,\min(p_i-1,p_j-1)\}$, and the operator
$$X_{i,j}:=\sum_{k=0}^{l}(\sum_{r=0}^{k}c_{k-r}(\frac{\alpha_i}{\alpha_j})^r
\frac{(l+r-k)!(p_i-r-1)!}{(l-k)!(p_j-1)!} e^*_{i,p_i-r-1})\otimes e_{j,l-k}$$
$$where\ c_{k-r}\in\mathbb{C},\ c_0\neq0\ and\ l=0,...,\min(p_i-1,p_j-1),\ 
,$$
is a nonzero solution of
\begin{equation}\label{q2}
 S_BX= \frac{\alpha_i}{\alpha_j}XS_B.
\end{equation}
Assume that $n$ is a natural number different from $j$ (i.e., $\alpha_n\neq\alpha_j$). Now, we find
the image of $e^*_{n,l}$ for $l=0,1,...,p_n-1$, under the operator $X$ that verify (\ref{q2}), hence
\begin{equation*}
 S_BXe^*_{n,l} = \left\{
  \begin{array}{l l}
    \frac{\alpha_i}{\alpha_j}\alpha_n Xe^*_{n,p_n-1} & \quad \text{if $l=p_n-1$}\\
    \frac{\alpha_i}{\alpha_j}\alpha_n Xe^*_{n,l}+\frac{\alpha_i}{\alpha_j}(l+1)Xe^*_{n,l+1} & \quad \text{if $l=0,...,p_n-2$.}\\
  \end{array} \right.
\end{equation*}

therefore, once again by Corollary \ref{vect}, if there is $l$ in $\{0,1,...,p_n-1\}$ such that $Xe^*_{n,l}\neq0$,
then necessarily there is a natural number $m$ (necessarily different from $i$) such that
$$\frac{\alpha_i}{\alpha_j}=\frac{\alpha_m}{\alpha_n},\  and\  Xe^*_{n,l}=c_{0}e^*_{m,p_m-1},
\  (c_{0}\neq0)\in\mathbb{C}.$$
So, in this case, $X$ has the same behavior like the $e^*_{j,l}$ case,
i.e., $X=X_{m,n}$ is a solution of (\ref{q2}).

Thus, we have exactly described the solution of (\ref{q2}) on a set which spans
the space $K_B^2$. Consequently, $E_{ext}(\alpha_i/\alpha_j)$ is given by
$$E_{ext}(\frac{\alpha_i}{\alpha_j})=
\Span\{X_{m,n},\ \forall m,n\geq1\ where\ \frac{\alpha_m}{\alpha_n}=\frac{\alpha_i}{\alpha_j}\},$$
as desired.

\end{proof}
\end{section}

 \begin{section}{Concluding remarks}
  We finish this paper with some remarks which are summarized in the following.
First, it is clear that Theorem \ref{moi1} is a particular case of last theorem, nevertheless we have proved it as 
a direct result of Theorem \ref{fini}.

In addition, if the set of zeros $\{\alpha_i\}_{i\geq1}$ satisfies the well-known Carleson condition (see \cite{nik}),
then the set $\{e^*_{i,l}\}$ forms a Riesz basis for $K_B^2$, and the solution of \ref{q2} is given in terms
of this basis and the dual Riesz basis $\{e_{i,l}\}$.

Also, if we suppose that $\alpha_0=0$ is a zero of $B$, then by using the proof of Theorem \ref{fini}, we have that 
$\s{S_B}=\mathbb{C}$. Indeed, the operator $X=e^*_{0,p_0-1}\otimes e_{0,0}$ satisfies that 
$S_BX=0=\lambda XS_B$, for all $\lambda$ in $\mathbb{C}$.

And finally, as a direct result of (2) in Corollary \ref{fini1}, if 
$$B(z)=b_{\alpha}^{n},\  where\  n\in\mathbb{N^*}\  and\  \alpha\in\mathbb{D},$$
then $\s{S_B}=\{1\}$ and $$E_{ext}(1)=\Span\{\sum_{k=0}^{l}(\sum_{r=0}^{k}c_{k-r}
\frac{(l+r-k)!(n-r-1)!}{(l-k)!(n-1)!} e^*_{\alpha,n-r-1})\otimes e_{\alpha,l-k},$$ $$ l=0,...,n-1\}.$$

Lastly, this paper gives a complete description  of the set of extended eigenvectors of $S_u$ in the case of $u$
is a Blaschke product, and this leads naturally to the following question
\end{section}
\begin{plm}
 What is the set of extended eigenvectors of $S_u$ in the case of $u$ is a singular inner function?
\end{plm}

\begin{ack}
 The author is grateful to Professor Gilles Cassier for his suggestion to this study and his help.
\end{ack}

\bibliographystyle{abbrv}
\bibliography{ref}

\def\cprime{$'$} \def\lfhook#1{\setbox0=\hbox{#1}{\ooalign{\hidewidth
  \lower1.5ex\hbox{'}\hidewidth\crcr\unhbox0}}}
\begin{thebibliography}{1}

\bibitem{berc}
H.~Bercovici.
\newblock {\em Operator theory and arithmetic in {$H^\infty$}}, volume~26 of
  {\em Mathematical Surveys and Monographs}.
\newblock American Mathematical Society, Providence, RI, 1988.

\bibitem{pfini}
A.~Biswas and S.~Petrovic.
\newblock On extended eigenvalues of operators.
\newblock {\em Integral Equations Operator Theory}, 55(2):233--248, 2006.

\bibitem{nik}
N.~K. Nikol{\cprime}ski{\u\i}.
\newblock {\em Treatise on the shift operator}, volume 273 of {\em Grundlehren
  der Mathematischen Wissenschaften [Fundamental Principles of Mathematical
  Sciences]}.
\newblock Springer-Verlag, Berlin, 1986.
\newblock Spectral function theory, With an appendix by S. V.
  Hru{\v{s}}{\v{c}}ev [S. V. Khrushch{\"e}v] and V. V. Peller, Translated from
  the Russian by Jaak Peetre.

\bibitem{ros}
M.~Rosenblum.
\newblock On the operator equation {$BX-XA=Q$}.
\newblock {\em Duke Math. J.}, 23:263--269, 1956.

\bibitem{don}
D.~Sarason.
\newblock Free interpolation in the {N}evanlinna class.
\newblock In {\em Linear and complex analysis}, volume 226 of {\em Amer. Math.
  Soc. Transl. Ser. 2}, pages 145--152. Amer. Math. Soc., Providence, RI, 2009.

\bibitem{foi}
B.~Sz.-Nagy and C.~Foia{\lfhook{s}}.
\newblock {\em Harmonic analysis of operators on {H}ilbert space}.
\newblock Translated from the French and revised. North-Holland Publishing Co.,
  Amsterdam, 1970.

\end{thebibliography}
\end{document}